\documentclass{article}

\usepackage[utf8]{inputenc}
\usepackage[T1]{fontenc}

\usepackage{amsmath,amssymb}
\usepackage{amsthm}
\usepackage[pdftex,bookmarks=true]{hyperref}
\usepackage{enumitem}
\usepackage[noadjust]{cite}
\usepackage[capitalise, nameinlink, noabbrev]{cleveref}
\usepackage{array}

\usepackage{tikz}
\tikzset{
every node/.style={draw, circle, inner sep=2pt}
}

\tikzset{
p4pic/.pic = {
    \node (1) at (1,0) {};
    \node (2) at (2,0) {};
    \node (3) at (3,0) {};
    \node (4) at (4,0) {};
    \draw (1) -- (2) -- (3) -- (4);
}, 
k13pic/.pic = {
    \node (1) at (0,0) {};
    \node (2) at (-1,0) {};
    \node (3) at (30:1) {};
    \node (4) at (-30:1) {};
    \draw (1) -- (2);
    \draw (1) -- (3);
    \draw (1) -- (4);
},
pawpic/.pic = {
    \node (1) at (0,0) {};
    \node (2) at (-1,0) {};
    \node (3) at (30:1) {};
    \node (4) at (-30:1) {};
    \draw (1) -- (2);
    \draw (1) -- (3) -- (4) -- (1);
}, 
c4pic/.pic = {
    \node (1) at (0,0) {};
    \node (2) at (1,0) {};
    \node (3) at (1,1) {};
    \node (4) at (0,1) {};
    \draw (1) -- (2) -- (3) -- (4) -- (1);
}, 
dmndpic/.pic = {
    \node (1) at (0,0) {};
    \node (2) at (1,0) {};
    \node (3) at (1,1) {};
    \node (4) at (0,1) {};
    \draw (1) -- (2) -- (3) -- (4) -- (1);
    \draw (2) -- (4);
}, 
k4pic/.pic = {
    \node (1) at (0,0) {};
    \node (2) at (1,0) {};
    \node (3) at (1,1) {};
    \node (4) at (0,1) {};
    \draw (1) -- (2) -- (3) -- (4) -- (1);
    \draw (1) -- (3);
    \draw (2) -- (4);
}, 
}
\tikzset{
p5pic/.pic = {
    \node (1) at (1,0) {};
    \node (2) at (2,0) {};
    \node (3) at (3,0) {};
    \node (4) at (4,0) {};
    \node (5) at (5,0) {};
    \draw (1) -- (2) -- (3) -- (4) -- (5);
}, 
s211pic/.pic = {
    \node (1) at (1,0) {};
    \node (2) at (2,0) {};
    \node (3) at (3,0) {};
    \node (4) at ([shift={(30:1)}]3) {};
    \node (5) at ([shift={(-30:1)}]3) {};
    \draw (1) -- (2) -- (3) -- (4);
    \draw (3) -- (5);
}, 
k14pic/.pic = {
    \node (1) at (0,0) {};
    \node (2) at (45:1) {};
    \node (3) at (135:1) {};
    \node (4) at (225:1) {};
    \node (5) at (315:1) {};
    \draw (1) -- (2);
    \draw (1) -- (3);
    \draw (1) -- (4);
    \draw (1) -- (5);
}, 
l32pic/.pic = {
    \node (1) at (1,0) {};
    \node (2) at (2,0) {};
    \node (3) at (3,0) {};
    \node (4) at ([shift={(30:1)}]3) {};
    \node (5) at ([shift={(-30:1)}]3) {};
    \draw (1) -- (2) -- (3) -- (4) -- (5) -- (3);
}, 
bullpic/.pic = {
    \node (1) at (1,0) {};
    \node (2) at (2,0) {};
    \node (3) at (3,0) {};
    \node (4) at (4,0) {};
    \node (5) at ([shift={(-60:1)}]2) {};
    \draw (1) -- (2) -- (3) -- (4);
    \draw (2) -- (5) -- (3);
}, 
c5pic/.pic = {
    \node (1) at (90:1) {};
    \node (2) at (162:1) {};
    \node (3) at (234:1) {};
    \node (4) at (306:1) {};
    \node (5) at (18:1) {};
    \draw (1) -- (2) -- (3) -- (4) -- (5) -- (1);
}, 
camppic/.pic = {
    \node (1) at (0,0) {};
    \node (2) at (45:1) {};
    \node (3) at (135:1) {};
    \node (4) at (225:1) {};
    \node (5) at (315:1) {};
    \draw (1) -- (2) -- (3) -- (1);
    \draw (1) -- (4);
    \draw (1) -- (5);
}, 
bnrpic/.pic = {
    \node (1) at (0,0) {};
    \node (2) at (1,0) {};
    \node (3) at ([shift={(45:1)}]2) {};
    \node (4) at ([shift={(-45:1)}]3) {};
    \node (5) at ([shift={(-45:1)}]2) {};
    \draw (1) -- (2) -- (3) -- (4) -- (5) -- (2);
}, 
dartpic/.pic = {
    \node (1) at (0,0) {};
    \node (2) at (1,0) {};
    \node (3) at ([shift={(45:1)}]2) {};
    \node (4) at ([shift={(-45:1)}]3) {};
    \node (5) at ([shift={(-45:1)}]2) {};
    \draw (1) -- (2) -- (3) -- (4) -- (5) -- (2) -- (4);
}, 
kitepic/.pic = {
    \node (1) at (0,0) {};
    \node (2) at (1,0) {};
    \node (3) at ([shift={(45:1)}]2) {};
    \node (4) at ([shift={(-45:1)}]3) {};
    \node (5) at ([shift={(-45:1)}]2) {};
    \draw (1) -- (2) -- (3) -- (4) -- (5) -- (2);
    \draw (3) -- (5);
}, 
hspic/.pic = {
    \node (1) at (0,0) {};
    \node (2) at (1,0) {};
    \node (3) at (1,1) {};
    \node (4) at (0,1) {};
    \node (5) at ([shift={(60:1)}]4) {};
    \draw (1) -- (2) -- (3) -- (4) -- (1);
    \draw (3) -- (5) -- (4);
}, 
gempic/.pic = {
    \node (1) at (-90:1) {};
    \node (2) at (-18:1) {};
    \node (3) at (54:1) {};
    \node (4) at (126:1) {};
    \node (5) at (198:1) {};
    \draw (1) -- (2) -- (3) -- (4) -- (5) -- (1);
    \draw (3) -- (1) -- (4);
}, 
bflypic/.pic = {
    \node (1) at (0,0) {};
    \node (2) at ([shift={(30:1)}]1) {};
    \node (3) at ([shift={(-30:1)}]1) {};
    \node (4) at ([shift={(150:1)}]1) {};
    \node (5) at ([shift={(210:1)}]1) {};
    \draw (1) -- (2) -- (3) -- (1);
    \draw (1) -- (4) -- (5) -- (1);
}, 
k23pic/.pic = {
    \node (1) at (0,0) {};
    \node (2) at (1,0) {};
    \node (3) at ([shift={(120:1)}]1) {};
    \node (4) at ([shift={(120:1)}]2) {};
    \node (5) at ([shift={(60:1)}]2) {};
    \draw (1) -- (3) -- (2) -- (5) -- (1);
    \draw (1) -- (4) -- (2);
}, 
t5pic/.pic = {
    \node (1) at (0,0) {};
    \node (2) at (1,0) {};
    \node (3) at ([shift={(120:1)}]1) {};
    \node (4) at ([shift={(120:1)}]2) {};
    \node (5) at ([shift={(60:1)}]2) {};
    \draw (1) -- (3) -- (2) -- (5) -- (1);
    \draw (1) -- (4) -- (2) -- (1);
}, 
l41pic/.pic = {
    \node (1) at (0,0) {};
    \node (2) at (1,0) {};
    \node (3) at ([shift={(45:1)}]2) {};
    \node (4) at ([shift={(-45:1)}]3) {};
    \node (5) at ([shift={(-45:1)}]2) {};
    \draw (1) -- (2) -- (3) -- (4) -- (5) -- (2) -- (4);
    \draw (3) -- (5);
}, 
k4epic/.pic = {
    \node (1) at (0,0) {};
    \node (2) at (1,0) {};
    \node (3) at (0,1) {};
    \node (4) at (-1,0) {};
    \node (5) at (0,-1) {};
    \draw (4) -- (1) -- (2);
    \draw (2) -- (3) -- (4) -- (5) -- (2);
    \draw (1) -- (3);
}, 
w5pic/.pic = {
    \node (1) at (0,0) {};
    \node (2) at (1,0) {};
    \node (3) at (0,1) {};
    \node (4) at (-1,0) {};
    \node (5) at (0,-1) {};
    \draw (4) -- (1) -- (2);
    \draw (2) -- (3) -- (4) -- (5) -- (2);
    \draw (5) -- (1) -- (3);
}, 
fhspic/.pic = {
    \node (1) at (0,0) {};
    \node (2) at (1,0) {};
    \node (3) at (1,1) {};
    \node (4) at (0,1) {};
    \node (5) at ([shift={(60:1)}]4) {};
    \draw (1) -- (2) -- (3) -- (4) -- (1) -- (3) -- (5) -- (4) -- (2);
}, 
k5epic/.pic = {
    \node (1) at (90:1) {};
    \node (2) at (162:1) {};
    \node (3) at (234:1) {};
    \node (4) at (306:1) {};
    \node (5) at (18:1) {};
    \draw (1) -- (2) -- (3) -- (4) -- (5) -- (1);
    \draw (2) -- (4) -- (1) -- (3) -- (5);
}, 
k5pic/.pic = {
    \node (1) at (90:1) {};
    \node (2) at (162:1) {};
    \node (3) at (234:1) {};
    \node (4) at (306:1) {};
    \node (5) at (18:1) {};
    \draw (1) -- (2) -- (3) -- (4) -- (5) -- (1);
    \draw (2) -- (4) -- (1) -- (3) -- (5) -- (2);
}, 
}

\usepackage{soul}
\usepackage{cancel}

\newtheorem{theorem}{Theorem}[section]
\newtheorem{lemma}[theorem]{Lemma}
\newtheorem{proposition}[theorem]{Proposition}
\newtheorem{corollary}[theorem]{Corollary}

\theoremstyle{definition}
\newtheorem{definition}[theorem]{Definition}
\newtheorem{observation}[theorem]{Observation}
\newtheorem{remark}[theorem]{Remark}
\newtheorem{example}[theorem]{Example}

\newcommand{\trans}{^\top}
\newcommand{\adj}{^{\rm adj}}

\newcommand{\inp}[2]{\left\langle#1,#2\right\rangle}
\newcommand{\dunion}{\mathbin{\dot\cup}}
\newcommand{\bzero}{\mathbf{0}}
\newcommand{\bone}{\mathbf{1}}
\newcommand{\ba}{\mathbf{a}}
\newcommand{\bb}{\mathbf{b}}

\newcommand{\bx}{\mathbf{x}}
\newcommand{\by}{\mathbf{y}}

\newcommand{\bu}{\mathbf{u}}
\newcommand{\bv}{\mathbf{v}}

\newcommand{\nul}{\operatorname{null}}

\newcommand{\Col}{\operatorname{Col}}

\newcommand{\spec}{\operatorname{spec}}
\newcommand{\vspan}{\operatorname{span}}

\newcommand{\mptn}{\mathcal{S}}
\newcommand{\mptncl}{\mathcal{S}^{\rm cl}}
\newcommand{\mptnclo}{\mathcal{S}^{\rm cl}_0}
\newcommand{\oml}{\mathbf{m}}
\newcommand{\mdso}{\ddot{\mathcal{S}}}
\newcommand{\mptnl}{\mathcal{S}_L}
\newcommand{\diag}{\operatorname{diag}}
\newcommand{\Paw}{\mathsf{Paw}}
\newcommand{\Dmnd}{\mathsf{Dmnd}}
\newcommand{\Bull}{\mathsf{Bull}}
\newcommand{\Camp}{\mathsf{Camp}}
\newcommand{\Bnr}{\mathsf{Bnr}}
\newcommand{\Dart}{\mathsf{Dart}}
\newcommand{\Kite}{\mathsf{Kite}}
\newcommand{\Hs}{\mathsf{Hs}}
\newcommand{\Gem}{\mathsf{Gem}}
\newcommand{\FHs}{\mathsf{FHs}}
\newcommand{\Bfly}{\mathsf{Bfly}}
\newcommand{\mtt}[2][]{\texttt{#2}\textsubscript{#1}}
\newcommand{\ntt}[2][]{\st{\texttt{#2}}\textsubscript{#1}}
\newcommand{\diam}{\operatorname{diam}}

\newcommand{\msym}[1][n]{\operatorname{Sym}_{#1}(\mathbb{R})}

\newcommand{\mult}{\operatorname{mult}}
\newcommand{\vsupp}{\operatorname{supp}}
\newcommand{\mvec}{\operatorname{vec}}

\newcommand{\putpic}[1]{\framebox{\tikz{\pic[scale=0.4, transform shape] at (0,0) {#1};}}}

\title{Inverse eigenvalue problem for discrete Schrödinger operators of a graph}
\author{
Anzila Laikhuram 
\thanks{Department of Applied Mathematics, National Sun Yat-sen University, Kaohsiung 80424, Taiwan (angi286@gmail.com)}
\and
Jephian C.-H.~Lin
\thanks{Department of Applied Mathematics, National Sun Yat-sen University, Kaohsiung 80424, Taiwan (jephianlin@gmail.com)}
}
\date{\today}

\begin{document}

\maketitle

\begin{abstract}
A discrete Schrödinger operator of a graph $G$ is a real symmetric matrix whose $i,j$-entry, $i \neq j$, is negative if $\{i,j\}$ is an edge and zero if it is not an edge, while diagonal entries can be any real numbers.  The discrete Schrödinger operators have been used to study vibration theory and the Colin de Verdi\`ere parameter.  The inverse eigenvalue problem for discrete Schr\"odinger operators of a graph aims to characterize the possible spectra among discrete Schrödinger operators of a graph. Compared to the inverse eigenvalue problem of a graph, the answers turn out to be more limited, and several restrictions based on graph structure are given.  Using the strong properties, analogous versions of the supergraph lemma, the liberation lemma, and the bifurcation lemma are established.  Using these results, the inverse eigenvalue problem for discrete Schr\"odinger operators is resolved for each graph with at most $5$ vertices.
\end{abstract}  


\medskip


\section{Introduction}

Let $G$ be a graph on $n$ vertices.  A \emph{generalized adjacency matrix} of $G$ is an $n\times n$ real symmetric matrix whose off-diagonal $i,j$-entry is nonzero whenever $\{i,j\}$ is an edge, where the diagonal entries can be any real number.  The set of all generalized adjacency matrices of $G$ is denoted as $\mptn(G)$.  The \emph{inverse eigenvalue problem of a graph $G$} (IEP-$G$) aims to characterize all possible spectra of matrices in $\mptn(G)$.  It has been an active research area.  We refer the reader to the monograph \cite{IEPGZF22} and the references therein.

A \emph{discrete Schrödinger operator} of $G$ is a matrix in $\mptn(G)$ whose off-diagonal entries are zero or negative \cite{BLS07, Goldber09}.  A \emph{weighted Laplacian matrix} of $G$ is a discrete Schrödinger operator $A$ with $A\bone = \bzero$ \cite{BLS07, IEPL}.  The set of discrete Schrödinger operators and the set of weighted Laplacian matrices of $G$ are denoted by $\mdso(G)$ and $\mptnl(G)$, respectively.  These families play important roles in vibration theory.

A graph $G$ can be viewed as a mass-spring system, where vertices are masses and edges are springs.  See, e.g., the monograph \cite{GladwellIPiV05} for more details.  Its vibration behavior is governed by the differential equation $M\ddot\bx = -L\bx$, where $M$ is a diagonal matrix determined by the weights of the masses, $\bx$ is a vector of displacements, and $L$ is a weighted Laplacian matrix based on the stiffness of the springs.  The solutions of this equation can be found through the eigenvalues and the eigenvectors of $M^{-\frac{1}{2}}LM^{-\frac{1}{2}}$, which is a discrete Schrödinger operator.

Due to the connection to vibration theory and physics, the inverse eigenvalue problem of a graph $G$ originally focuses on discrete Schrödinger operators.  Let $P_n$ and $C_n$ be the path and cycle on $n$ vertices, respectively.  Hochstadt \cite{Hochstadt74} called a matrix $A\in\mptn(P_n)$ a \emph{Jacobi matrix} if its off-diagonal entries are zero or positive; that is $-A\in\mdso(P_n)$.  Hochstadt showed that $\spec(A)$ and $\spec(A(1))$ have to be strictly interlacing, and any strictly interlacing $\spec(A)$ and $\spec(A(1))$ lead to a unique Jacobi matrix $A$, if exists.  Here $A(1)$ stands for the principal submatrix of $A$ by removing the first row and column.  After that, Hald~\cite{Hald76} and Gray and Wilson~\cite{GW76} independently introduced algorithms to construct the Jacobi matrix with prescribed spectral data.  In 1980, Ferguson~\cite{Ferguson80} called a matrix $A\in\mptn(C_n)$ a \emph{periodic Jacobi matrix} if its off-diagonal entries are zero or positive; that is $-A\in\mdso(C_n)$.  Ferguson characterized all possible spectra of a periodic Jacobi matrix and introduced an algorithm for constructing the matrix from prescribed spectral data.  Moreover, Boley and Golub~\cite{BG87} surveyed and studied the inverse eigenvalue problems of the form $-A$ with $A\in\mdso(G)$ for more graph families.  These fruitful results laid the groundwork for this inverse eigenvalue problem.  However, it appears that the attention of the IEP-$G$ has been shifted to $\mptn(G)$, and the research on $\mdso(G)$ seems discontinued.

On the other hand, Colin de Verdière \cite{CdV, CdVF} defined a graph parameter $\mu(G)$ as the maximum multiplicity of the second smallest eigenvalue among all discrete Schrödinger operators of $G$ with some non-degeneracy condition called the strong Arnold property.  Despite its linear algebraic definition, $\mu(G)$ is proved to possess many nice topological behaviors.  For example, it is minor monotone, which means $\mu(G) \leq \mu(H)$ if $G$ is a minor of $H$.  Also, it is shown that $\mu(G) \leq 3$ if and only if $G$ is a planar graph.  See \cite{HLS99} for a survey.

Motivated by these groundbreaking studies, this paper investigates the \emph{inverse eigenvalue problem for the discrete Schrödinger operators of a graph} (IEPS), which aims to characterize all possible spectra among matrices in $\mdso(G)$.  The rest of this section is used to discuss basic properties of $\mdso(G)$ and summarize known results.  \cref{sec:families} introduces new results about $\mdso(G)$ in several graph families.  \cref{sec:SSP} revisits the strong spectral property and establishes the analogues versions of the supergraph lemma, the liberation lemma, and the bifurcation lemma for $\mdso(G)$.  Moreover, we establish \cref{cor:firework} allowing us to increase each multiplicity by one, which works only for $\mdso(G)$ but not $\mptn(G)$.  Finally, we resolve the IEPS for all graphs with at most $5$ vertices in \cref{sec:small}.  

\subsection{Preliminaries}

We first introduce some notation to be used throughout the paper.  Let $A$ be a matrix.  We use $A[\alpha,\beta]$ for the submatrix of $A$ induced on rows in $\alpha$ and columns in $\beta$, and $A(\alpha,\beta)$ for the submatrix of $A$ obtained by removing rows in $\alpha$ and columns in $\beta$.  When $\alpha = \beta$, we use $A[\alpha] = A[\alpha,\alpha]$ and $A(\alpha) = A(\alpha,\alpha)$ as the shorthand.  When a graph $H$ is used as an index set, it means the vertex set $V(H)$.  We use the notation $\spec(A) = \{\lambda_1^{(m_1)}, \ldots, \lambda_q^{(m_q)}\}$ for the spectrum of $A$, where each of the distinct eigenvalues $\lambda_i$ has multiplicity $m_i$.  When $\lambda_1 < \cdots < \lambda_q$, we say $\mathtt{m_1\cdots m_q}$ is the \emph{ordered multiplicity list} of $A$.  We say a spectrum or an ordered multiplicity list is \emph{allowed} in $\mptn(G)$ if there is a matrix $A\in\mptn(G)$ with the given spectral property; otherwise, we say it is \emph{forbidden}.  An ordered multiplicity list is said to be \emph{spectrally arbitrary} if any spectrum with the ordered multiplicity list is allowed in the same set of matrices.  The definition naturally extends to $\mdso(G)$ and other set of matrices.  Finally, we say an eigenvalue is \emph{simple} if its multiplicity is one, while it is \emph{multiple} if its multiplicity is more than one.

For two graphs $G$ and $H$, $G\dunion H$ stands for the disjoint union of $G$ and $H$, and $G\vee H$ stands for the join of $G$ and $H$.  For a vertex $i$, $N_G(i)$ stands for the neighborhood of $i$ in $G$.   When the context is clear, $N(i)$ is used instead.

Matrices in $\mdso(G)$ demonstrates several special properties, in contrast to matrices in $\mptn(G)$.  Here we summarize some immediate ones.

Let $G$ be a connected graph and $A\in\mdso(G)$.  Then $\lambda I - A$ is a nonnegative matrix when $\lambda$ is large enough.  By the Perron--Frobenius theorem (see, e.g., \cite[Theorem~2.2.1]{BHSoG12}), the largest eigenvalue of $\lambda I - A$ is simple and the corresponding eigenvector can be chosen to be entrywise positive.  Equivalently, we have the following observation, which is one of the major differences between the spectral behaviors of matrices in $\mdso(G)$ and that in $\mptn(G)$.  

\begin{observation}
\label{obs:pf}
Let $G$ be a connected graph and $A\in\mdso(G)$. Then the smallest eigenvalue of $A$ is simple, and the corresponding eigenvector can be chosen to be entrywise positive.  
\end{observation}

On the other hand, matrices in $\mdso(G)$ and the adjacency matrix share the same diameter bound for the number of distinct eigenvalues.  Let $q(A)$ be the number of distinct eigenvalues of $A$.  Let $\diam(G)$ be the diameter of $G$, that is, the largest distance between two vertices.  It is known that $\diam(G) + 1 \leq q(A)$ if $A$ is the adjacency matrix of $G$; see, e.g., \cite{BHSoG12}.  In fact, the same proof also works for any matrices in $\mdso(G)$.  

\begin{theorem}
\label{thm:diam}
Let $G$ be a connected graph and $A\in\mdso(G)$.  Then $\diam(G) + 1 \leq q(A)$.   
\end{theorem}

\begin{example}
\label{ex:kited}
Let $\Kite$ be the graph obtained from $C_4$ by adding a chord and appending a leaf on a vertex of degree $2$; see \cref{tbl:graph5wo3}.  According to \cite{IEPG2}, the ordered multiplicity list \mtt{122} is allowed in $\mptn(\Kite)$.  However, by \cref{thm:diam}, $4 = \diam(\Kite) + 1 \leq q(A)$ for any $A\in\mdso(\Kite)$.  Therefore, \mtt{122} is forbidden in $\mdso(\Kite)$.  
\end{example}

Next, we justify that $\mdso(G)$ is indeed the right focus in studying the inverse problem on vibration.  In the equation $M\ddot\bx = -L\bx$, $M$ is a diagonal matrix whose diagonal entries are positive, so the equation is equivalent to $\ddot\by = -M^{-\frac{1}{2}}LM^{-\frac{1}{2}}\by$, where $\by = M^{\frac{1}{2}}\bx$.  Therefore, it is natural to focus on matrices of the form $M^{-\frac{1}{2}} L M^{-\frac{1}{2}}$ such that $L\in\mptnl(G)$ and $M$ is a diagonal matrix with positive diagonal entries.  It is well known that the weighted Laplacian matrix $L$ of a connected graph is positive semidefinite with $\nul(A) = 1$; see, e.g., \cite{BapatGM14, BLS07}.  Therefore, every matrix of the form $M^{-\frac{1}{2}} L M^{-\frac{1}{2}}$ is congruent to $L$ and hence a singular positive semidefinite matrix in $\mdso(G)$.  

On the other hand, let $A$ be a singular positive semidefinite matrix in $\mdso(G)$ for a connected graph $G$.  By \cref{obs:pf}, $\nul(A) = 1$ and the eigenvector of $A$ with respect to $0$ can be chosen to be a positive vector $\bv$.  By choosing $M^{\frac{1}{2}} = \diag(\bv)$ as the diagonal matrix whose diagonal entries are $\bv$, or equivalently, $M = \diag(\bv^{\circ 2})$, where $\bv^{\circ 2}$ is the vector obtained from $\bv$ by taking the square on each entry, we may easily compute 
\[
    M^{\frac{1}{2}} A M^{\frac{1}{2}} \bone = M^{\frac{1}{2}} A \bv = \bzero.
\]
Thus, $L = M^{\frac{1}{2}} A M^{\frac{1}{2}}$ is a matrix in $\mptnl(G)$.  

\begin{observation}
\label{obs:mlm}
Let $G$ be a connected graph.  Then every singular positive semidefinite matrix $A\in\mdso(G)$ has a unique decomposition $M^{-\frac{1}{2}}LM^{-\frac{1}{2}}$ such that $L\in\mptnl(G)$ and $M$ is a diagonal matrix whose diagonal entries are positive.
\end{observation}

In general, for every matrix $A\in\mdso(G)$, $A - \lambda_{\min}I$ is a singular positive semidefinite matrix, where $\lambda_{\min}$ is the smallest eigenvalues of $A$.  Therefore, solving the inverse eigenvalue problem in $\mdso(G)$ and solving the inverse eigenvalue problem for singular positive semidefinite matrices in $\mdso(G)$ are equivalent.

In fact, early results on the inverse problem focus more on $\mdso(G)$.  Hald~\cite{Hald76} and Gray and Wilson~\cite{GW76} independently solved the IEPS for paths.

\begin{theorem}[\cite{GW76,Hald76}]
\label{thm:pn}
Let $\Lambda = \{\lambda_1, \ldots, \lambda_n\}$ be a multiset of real numbers.  Then $\Lambda = \spec(A)$ for some $A\in\mdso(P_n)$ if and only if elements in $\Lambda$ are distinct.
\end{theorem}

On the other hand, Ferguson~\cite{Ferguson80} solved the IEPS for cycles using a discrete version of Floquet theory.

\begin{theorem}[\cite{Ferguson80}]
\label{thm:cn}
Let $\Lambda = \{\lambda_1, \ldots, \lambda_n\}$ be a multiset of real numbers such that $\lambda_1 \leq \cdots \leq \lambda_n$.  Then $\Lambda = \spec(A)$ for some $A\in\mdso(C_n)$ if and only if 
\[
    \lambda_1 < \lambda_2 \leq \lambda_3 < \lambda_4 \leq \lambda_5 < \cdots .
\]
\end{theorem}

These results are based on algorithmic constructions of a matrix from given spectral data.  Recent development of the IEP-$G$ utilizes the strong properties to construct new matrices implicitly.  

A matrix $A$ is said to have the \emph{strong spectral property} (SSP) if $X = O$ is the only real symmetric matrix that satisfies $A\circ X = O$, $I\circ X = O$, and $AX - XA = O$.  The SSP was introduced in \cite{gSAP} and motivated by the strong Arnold property used in the Colin de Verdi\`ere parameter \cite{CdV, CdVF}.  These strong properties have been powerful tools in constructing new matrices with desired spectral properties \cite{gSAP, IEPG2, IEPGZF22}.  

\begin{remark}
\label{rem:ssppc}
Note that \cref{thm:pn,thm:cn} did not concern about the SSP.  However, \cite{Gssp} showed that every matrix in $\mptn(P_n)$ has the SSP, and so are those in $\mdso(P_n)$.  In contrast, some matrices in $\mdso(C_n)$ do not have the SSP.
\end{remark}

We first review some results using the SSP.  Here we call $H$ a \emph{spanning supergraph} of $G$ if $H$ is a supergraph of $G$ and $V(H) = V(G)$.  

\begin{theorem}[Supergraph lemma~\cite{gSAP}]
\label{thm:super}
Let $G$ be a graph and $H$ a spanning supergraph of $G$.  If $A\in\mptn(G)$ has the SSP, then there is a matrix $A'\in\mptn(H)$ with the SSP such that $\spec(A') = \spec(A)$.  Moreover, $A'$ can be chosen so that $\|A' - A\|$ is arbitrarily small.
\end{theorem}

A follow-up paper \cite{IEPG2} generalized the supergraph lemma into the liberation lemma, which works for matrices without the SSP\@.  Here $\vsupp(\bx)$ is the index set on which $\bx$ is nonzero.  The SSP verification matrix $\Psi$ of $A$ is a matrix whose rows are indexed by $E(\overline{G})$, which will be introduced in \cref{sec:SSP}.  

\begin{theorem}[Liberation lemma~\cite{IEPG2}]
\label{thm:lib}
Let $G$ be a graph, $A\in\mptn(G)$, and $\Psi$ the SSP verification matrix of $A$.  Suppose $\bx$ is a vector in the column space $\Col(\Psi)$ such that rows of $\Psi$ outside $\vsupp(\bx)$ are linearly independent.  Then there is a matrix $A'\in\mptn(G + \vsupp(\bx))$ with the SSP such that $\spec(A') = \spec(A)$.  Moreover, $A'$ can be chosen so that $\|A' - A\|$ is arbitrarily small.
\end{theorem}

While the supergraph lemma and liberation lemma guarantee a matrix with the same spectrum and a nearby pattern, the bifurcation lemma guarantees a matrix with the same pattern and a nearby spectrum.  

\begin{theorem}[Bifurcation lemma~\cite{bifur}]
\label{thm:bif}
Let $G$ be a graph and $A\in\mptn(G)$ with the SSP.  Then there is $\epsilon > 0$ such that for any $M$ with $\|M - A\| < \epsilon$, a matrix $A'\in\mptn(G)$ exists with the SSP and $\spec(A') = \spec(M)$.
\end{theorem}

There are yet more applications of the strong properties; see, e.g., \cite{IEPGZF22}.  In \cref{sec:SSP}, we will introduce the analogous versions of these lemmas for the discrete Schrödinger operators.  

\section{Graph families}
\label{sec:families}

In this section, we study the spectral behaviors of the discrete Schrödinger operators for several graph families.  We start with complete graphs and trees.

\subsection{Complete graphs}

The \emph{Givens rotation matrix} $R_{i,j;\theta}$ of order $n$ is obtained from the identity matrix of order $n$ by replacing the principal submatrix induced on $i$ and $j$ (in order) with the $2$-dimensional rotation matrix
\begin{equation}
\label{eq:rtheta}
    R_\theta = \begin{bmatrix}
        \cos\theta & -\sin\theta \\
        \sin\theta & \cos\theta
    \end{bmatrix}, 
\end{equation}
where $1\leq i,j\leq n$ are the distinct indices corresponding to the axes of the rotation plane and $\theta\in\mathbb{R}$ is the rotation angle.

The Givens rotations have been used to solve the inverse eigenvalue problem for complete graphs \cite{BLMNSSSY13}.  We follow the similar proof but pay extra attention to make sure the off-diagonal entries are negative.

\begin{lemma}
\label{lem:givens}
Let $G$ be a graph with two vertices $i$ and $j$ such that $N(i)\setminus\{j\} \supseteq N(j)\setminus\{i\}$.  Define $H$ as the graph obtained from $G$ by adding edges from $j$ to each vertex in $N(i)\cup\{i\}\setminus\{j\}$ if the edge was not present.  If $A = \begin{bmatrix} a_{i,j} \end{bmatrix} \in \mdso(G)$ such that either $i \sim j$ in $G$ or $a_{i,i} < a_{j,j}$ when $i \nsim j$ in $G$, then the matrix $A' = R_{i,j;\theta} A R_{i,j;\theta}\trans$ is a matrix in $\mdso(H)$ that is similar to $A$ for any $\theta > 0$ small enough.
\end{lemma}
\begin{proof}
Without loss of generality, we may assume $i = 1$ and $j = 2$.  Let $n$ be the order of $G$.  Thus, we may write 
\[
    A = \begin{bmatrix}
        A_{1,1} & A_{1,2} \\
        A_{1,2}\trans & A_{2,2}
    \end{bmatrix}
    \text{ and }
    R_{i,j;\theta} = \begin{bmatrix}
        R_\theta & O_{2,n-2} \\
        O_{n-2,2} & I_{n-2}
    \end{bmatrix}
\]
such that $A_{1,1}$ is of order $2$.  Consequently, 
\[
    R_{i,j;\theta} A R_{i,j;\theta}\trans = \begin{bmatrix}
        R_\theta A_{1,1} R_\theta\trans & R_\theta A_{1,2} \\
        A_{1,2}\trans R_\theta\trans & A_{2,2}
    \end{bmatrix}.
\]
Since $N(i)\setminus\{j\} \supseteq N(j)\setminus\{i\}$, each nonzero column $\bv$ in $A_{1,2}$ is of the form
\[
    \begin{bmatrix} - \\ - \end{bmatrix}
    \text{ or } 
    \begin{bmatrix} - \\ 0 \end{bmatrix},
\] where each ``$-$'' represents a negative real number.
Thus, for $\theta$ small enough, $R_\theta \bv$ is a entrywise negative vector.  

Suppose $i \sim j$.  Then the fact that $A_{1,1}\in\mdso(K_2)$ implies $R_\theta A_{1,1} R_\theta\trans$ is in $\mdso(K_2)$ when $\theta$ is small enough.  Suppose $i \nsim j$ and $a_{i,i} < a_{j,j}$.  Then by direct computation
\[
    \begin{aligned}
    R_\theta A_{1,1} R_\theta\trans &= 
    \begin{bmatrix} \cos\theta & -\sin\theta \\ \sin\theta & \cos\theta \end{bmatrix}
    \begin{bmatrix} a_{i,i} & 0 \\ 0 & a_{j,j} \end{bmatrix}
    \begin{bmatrix} \cos\theta & \sin\theta \\ -\sin\theta & \cos\theta \end{bmatrix} \\
    &= \begin{bmatrix}
        a_{i,i}\cos^2\theta + a_{j,j}\sin^2\theta & (a_{i,i} - a_{j,j})\sin\theta\cos\theta \\
        (a_{i,i} - a_{j,j})\sin\theta\cos\theta & a_{i,i}\sin^2\theta + a_{j,j}\cos^2\theta. 
    \end{bmatrix}
    \end{aligned}
\]
With $a_{i,i} < a_{j,j}$ and $\sin\theta\cos\theta = \frac{1}{2}\sin 2\theta$, the matrix $R_\theta A_{1,1} R_\theta\trans$ is in $\mdso(K_2)$ when $\theta > 0$ is small enough.
\end{proof}

Roughly speaking, the idea of \cref{lem:givens} is that, if the $i$-th row is ``more negative'' than the $j$-th row, then the Givens rotation on $i,j$ with a $\theta > 0$ small enough leads to another discrete Schrödinger operator through similarity.  

\begin{theorem}
\label{thm:kn}
Let $\Lambda = \{\lambda_1, \lambda_2, \ldots, \lambda_n\}$ be a multiset of real numbers with $\lambda_1 \leq \lambda_2 \leq \cdots \leq \lambda_n$.  Then $\Lambda$ is the spectrum of some matrix $A\in\mdso(K_n)$ if and only if $\lambda_1 < \lambda_2$.
\end{theorem}
\begin{proof}
Let $D$ be the diagonal matrix whose $i,i$-entry is $\lambda_i$.  Then $\spec(D) = \Lambda$.  By applying \cref{lem:givens} with $i = 1$ and $j = 2, \ldots, n$ sequentially, there is a matrix $A\in\mdso(K_n)$ with $\spec(A) = \Lambda$.  Note that we can always choose $\theta > 0$ small enough to make sure that $a_{1,1}$ is always the unique minimum element on the diagonal. 
\end{proof}

\subsection{Trees}

The signature similarity techniques have been used in many places in the inverse eigenvalue problem \cite{loopEGR, IEPG2} and sign pattern problem \cite{DHHHW06}.  

A \emph{signature matrix} is a diagonal matrix whose diagonal entries are either $1$ or $-1$.  Note that the inverse of a signature matrix is itself.  If $A$ and $B$ are two matrices such that $B = DAD$ for some signature matrix $D$, then we say $A$ and $B$ are \emph{signature similar}.

\begin{proposition}[\cite{DHHHW06}]
\label{prop:sigsim}
Let $G$ be a connected graph and $A\in\mptn(G)$.  For any spanning tree $T$ of $G$, there is a signature matrix $D$ such that $DAD$ is a matrix in $\mptn(G)$ whose entries corresponding to $E(T)$ are negative.
\end{proposition}

Consequently, the standard inverse eigenvalue problems on trees are equivalent to the inverse eigenvalue problems of discrete Schrödinger operators of trees.
\cref{thm:iepstree} follows immediately from \cref{prop:sigsim}, since signature similarity preserves the spectrum and every matrix in $\mptn(T)$ is signature similar to a matrix in $\mdso(T)$.

\begin{theorem}
\label{thm:iepstree}
Let $T$ be a tree and $\Lambda$ a multiset of real numbers.  Then $\Lambda$ is the spectrum of some matrix in $\mdso(T)$ if and only if $\Lambda$ is the spectrum of some matrix in $\mptn(T)$.  
\end{theorem}

The IEP-$G$ is solved for path \cite{GW76,Hald76}, generalized stars and double stars \cite{JLDS03}, and linear trees \cite{JW22}, but remains open for general trees.

\subsection{Unicyclic graphs}

\begin{theorem}
\label{thm:unico}
Let $G$ be a unicyclic graph whose cycle has odd length.  If a spectrum is allowed in $\mptn(G)$ such that the largest eigenvalue is multiple, then the spectrum is allowed in $\mdso(G)$.  
\end{theorem}
\begin{proof}
Let $A\in\mptn(G)$.  Let $p$ be the product of all the entries of $A$ above the diagonal corresponding to the unique cycle.  We claim that if $p < 0$, then $A$ is signature similar to a matrix in $\mdso(G)$.  To see this, suppose $p < 0$.  By \cref{prop:sigsim}, $A$ is signature similar to a matrix $B$ whose off-diagonal entries corresponding to all edges are negative, except for possibly one edge $e$ on the cycle.  
Note that any signature similarity does not change the product $p$, so the product for $B$ is also $p < 0$.  Since the length of the cycle is odd, $p < 0$ implies the pair of entries in $B$ corresponding to the edge $e$ are negative, so $B\in\mdso(G)$.

For the case when $p > 0$, we know $-A\in\mptn(G)$ has its product negative since the length of the cycle is odd.  By the same argument, $-A$ is signature similar to a matrix in $\mdso(G)$.

In summary, for any matrix $A\in\mptn(G)$, either $A$ or $-A$ is signature similar to $\mdso(G)$.  If the spectrum of $A$ has the largest eigenvalue multiple, then $-A$ cannot be signature similar to a matrix in $\mdso(G)$ by \cref{obs:pf}.  Therefore, $A$ is signature similar to a matrix in $\mdso(G)$, and the same spectrum is allowed in $\mdso(G)$. 
\end{proof}

\begin{example}
\label{ex:unic5}
Let $C_5$ be the cycle graph on $5$ vertices.  According to \cite{IEPG2}, $\mptn(C_5)$ allows the ordered multiplicity list \mtt{122}, and it is spectrally arbitrary.  By \cref{thm:unico}, any matrix $A\in\mptn(G)$ realizing \mtt{122} is signature similar to a matrix $A'\in\mdso(C_5)$.  Therefore, $\mdso(C_5)$ allows \mtt{122} as well, and it is spectrally arbitrary.  Let $\Camp$ be the graph obtained from $K_3$ by appending two leaves at the same vertex; see \cref{tbl:graph5wo3}.  A similar argument shows that $\mdso(\Camp)$ allows \mtt{122} as well, and it is spectrally arbitrary.
\end{example}

\subsection{Bipartite graphs}

\begin{theorem}
\label{thm:bip}
Let $G$ be a connected bipartite graph.  Then for any $A\in\mdso(G)$, its largest eigenvalue has to be simple.
\end{theorem}
\begin{proof}
Let $A\in\mdso(G)$.  Suppose $V(G) = X\dunion Y$ certifies that $G$ is a bipartite graph.  Let $D$ be a diagonal matrix of order $|V(G)|$ whose $i,i$-entry is $1$ if $i\in X$ and $-1$ if $i\in Y$.  Then $-A$ is similar to $-DAD\in\mdso(G)$, so the smallest eigenvalue of $-A$ and the largest eigenvalue of $A$ are both simple.  
\end{proof}

\begin{example}
\label{ex:k23bip}
According to \cite{IEPG2}, the complete bipartite graph $K_{2,3}$ has the ordered multiplicity lists \mtt{1112} and \mtt{113} allowed in $\mptn(K_{2,3})$ and they are both spectrally arbitrary.  However, these two ordered multiplicity lists are forbidden for matrices in $\mdso(K_{2,3})$ by \cref{thm:bip}. 
\end{example}

\begin{corollary}
\label{cor:bipv}
Let $G$ be a graph.  If $G$ becomes a connected bipartite graph after the removal of $k$ vertices, then the multiplicity of the largest eigenvalue for any $A\in\mdso(G)$ is at most $k + 1$.  
\end{corollary}
\begin{proof}
Let $A\in\mdso(G)$ and $S$ a set of $k$ vertices such that $G - S$ is a connected bipartite graph.  Let $\lambda$ be the largest eigenvalue of $A$ with multiplicity $m$.  For the purpose of yielding the contradiction, suppose $m \geq k + 2$.  Then $A(S)$ also has $\lambda$ as its largest eigenvalue by the Cauchy interlacing theorem (see, e.g., \cite{BHSoG12}) with $\mult_{A(S)}(\lambda) \geq m - k \geq 2$.  However, this leads to a contradiction since $\mult_{A(S)}(\lambda) \leq 1$ by \cref{thm:bip}.  Therefore, the assumption is invalid and $m \leq k + 1$.
\end{proof}

\begin{example}
\label{ex:k4se}
Let $(K_4)_e$ be the graph obtained from $K_4$ by subdividing an edge; see \cref{tbl:graph5w3}.  By removing a vertex of degree $3$, it becomes a connected bipartite graph.  Therefore, for any matrix in $\mdso((K_4)_e)$, the multiplicity of the largest eigenvalue is at most $1 + 1 = 2$.  Thus, the ordered multiplicity list \mtt{113} is forbidden in $\mdso((K_4)_e)$.  However, it is allowed in $\mptn((K_4)_e)$ according to \cite{IEPG2}.  Similarly, the wheel graph $W_5$ on $5$ vertices becomes a connected bipartite graph after the removal of the center vertex, so \mtt{113} is forbidden for $\mdso(W_5)$.  
\end{example}

\subsection{Graphs with a cut-vertex}

Let $A$ be a real symmetric matrix and $v$ an index.  It is known that 
\[
    |\nul(A) - \nul(A(v))| \leq 1.
\]
Equivalently, for any $\lambda\in\mathbb{R}$, 
\[
    |\mult_A(\lambda) - \mult_{A(v)}(\lambda)| \leq 1.
\]
Here we consider the multiplicity as $0$ if $\lambda$ is not an eigenvalue.  Many works have been done to study the change of the nullity or the multiplicity after removal of a row and a column; see, e.g., \cite{P1960, W1984, JDS2003, BFH04, SNIP}.  

\begin{definition}
\label{def:und}
Let $A$ be a real symmetric matrix whose columns and rows are indexed by $V(G)$ for some graph $G$.  Let $\lambda\in\mathbb{R}$ and $v\in V(G)$.  Then $v$ is called 
\begin{itemize}
\item a \emph{$\lambda$-upper} vertex if $\mult_{A(v)}(\lambda) = \mult_A(\lambda) + 1$, 
\item a \emph{$\lambda$-neutral} vertex if $\mult_{A(v)}(\lambda) = \mult_A(\lambda)$, and 
\item a \emph{$\lambda$-downer} vertex if $\mult_{A(v)}(\lambda) = \mult_A(\lambda) - 1$. 
\end{itemize}
\end{definition}

\begin{lemma}
\label{lem:minup}
Let $G$ be a connected graph and $v$ a cut-vertex of $G$.  Let $A\in\mdso(G)$.  Then $v$ has to be a $\lambda$-upper vertex if $\lambda$ is the smallest eigenvalue of $A[H]$ for some component $H$ of $G - v$.  
\end{lemma}
\begin{proof}
Let $H$ be a component of $G - v$ such that $\lambda$ is the smallest eigenvalue of $A[H]$.  By \cref{obs:pf}, $\lambda$ is simple and the eigenvector $\bu$ of $A[H]$ with respect to $\lambda$ can be chosen to be entrywise positive.  Thus, $A[H,v]$ cannot be orthogonal to $\bu$ due to their signs.  Equivalently, $A[H,v]$ is not in the column space of $A[H] - \lambda I$, so $v$ has to be a $\lambda$-upper vertex.
\end{proof}

\begin{example}
\label{ex:bfly}
Let $\Bfly$ be the the graph $K_1\vee(2K_2)$ and $v$ the dominating vertex; see \cref{tbl:graph5wo3}.  We claim that \mtt{131} is forbidden in $\mdso(\Bfly)$.  For the purpose of yielding a contradiction, suppose $A\in\mdso(\Bfly)$ is a matrix with $\spec(A) = \{\lambda_1, \lambda_2^{(3)}, \lambda_3\}$ and $\lambda_1 < \lambda_2 < \lambda_3$.  Since $\Bfly - v$ is composed of two paths, $\mult_{A(v)}(\lambda_2) \leq 2$ by \cref{thm:pn}, so $v$ cannot be a $\lambda_2$-upper vertex.  By the Cauchy interlacing theorem (see, e.g., \cite{BHSoG12}), $\spec(A(v)) = \{\mu_1, \mu_2^{(2)}, \mu_3\}$ with $\mu_1 < \mu_2 < \mu_3$ and $\mu_2 = \lambda_2$. 
Again by \cref{thm:pn}, each path corresponds to two distinct eigenvalues in $A(v)$, so at least one path $P$ has $\mu_2$ as the smallest eigenvalue in $A[P]$, which leads to a contradiction to \cref{lem:minup}.  Therefore, \mtt{131} is forbidden in $\mdso(\Bfly)$.   
\end{example}

\begin{lemma}
\label{lem:fulljoin}
Let $G$ be a connected graph and $v$ a cut-vertex of $G$.  Let $A\in\mdso(G)$.  If $H$ is a component of $G - v$ such that $v$ is not a $\lambda$-upper vertex 
for any $\lambda\in\spec(A[H])$ that is not smallest, then $v$ is adjacent to every vertex of $H$.
\end{lemma}
\begin{proof}
Let $\Lambda$ be the set of values in $\spec(A[H])$ except for the smallest eigenvalue.   Since $v$ is not a $\lambda$-upper vertex for all $\lambda\in\Lambda$, $A[H,v]$ is in the column space of $A[H] - \lambda I$ for all $\lambda\in\Lambda$.  Equivalently, $A[H,v]$ is orthogonal to any eigenvector of $A[H]$ with respect to $\lambda\in\Lambda$, so $A[H,v]$ is parallel to the eigenvector $\bu$ of the smallest eigenvalue of $A[H]$.  By \cref{obs:pf}, $\bu$ can be chosen to be entrywise negative.  Thus, $A[H,v]$ is entrywise negative and $v$ is adjacent to every vertex of $H$.  
\end{proof}

Although the conclusion of the next example is the same as \cref{ex:kited}, it shows a potential that \cref{lem:fulljoin} along with \cref{lem:minup} gives the detailed structure of the eigenvalues that reside in each component. 

\begin{example}
\label{ex:kitej}
Let $\Kite$ be the graph obtained from $C_4$ by adding a chord and appending a leaf on a vertex of degree $2$; see \cref{tbl:graph5wo3}.  Let $v$ be the cut vertex.  We claim that \mtt{122} is forbidden in $\mdso(\Kite)$.  For the purpose of yielding a contradiction, let $A\in\mdso(\Kite)$ be a matrix with $\spec(A) = \{\lambda_1, \lambda_2^{(2)}, \lambda_3^{(2)}\}$ and $\lambda_1 < \lambda_2 < \lambda_3$.  By the Cauchy interlacing theorem (see, e.g., \cite{BHSoG12}), we have $\spec(A(v)) = \{\mu_1, \mu_2, \mu_3, \mu_4\}$ with $\mu_1 \leq \cdots \leq \mu_4$, $\mu_2 = \lambda_2$, and $\mu_4 = \lambda_3$.  Note that $\mu_1 = \mu_2 = \mu_3 < \mu_4$ is impossible since $\Kite - v$ has only two components, namely $P_1$ and $K_3$, and the smallest eigenvalue has multiplicity at most $2$ by \cref{obs:pf}.  Thus, the multiplicity of each eigenvalue of $A(v)$ is at most $2$.  By \cref{lem:minup}, $\mu_2$ and $\mu_4$ cannot be the smallest eigenvalue corresponding to the $P_1$ or the $K_3$.  The only possible arrangement is that the $P_1$ takes $\mu_3$ and the $K_3$ takes $\mu_1, \mu_2, \mu_4$.  (Moreover, $\mu_1$ and $\mu_3$ are the smallest eigenvalues in the two components, which have to be simple, so they cannot be $\mu_2$ nor $\mu_4$.)  Thus, the eigenvalues in $K_3$ that are not smallest are $\mu_2$ and $\mu_4$, and $v$ is not $\mu_2$-upper nor $\mu_4$-upper.  By \cref{lem:fulljoin}, $v$ is joined to every vertex in $K_3$, which is a contradiction.
\end{example}

\section{Strong spectral property}
\label{sec:SSP}

Recall that a real symmetric matrix $A$ is said to have the SSP if $X = O$ is the only real symmetric matrix that satisfies $A\circ X = O$, $I\circ X = O$, and $AX - XA = O$.

We first check that signature similarity preserves the SSP.

\begin{proposition}
Let $D$ be a signature matrix.  Then $A$ has the SSP if and only if $DAD$ has the SSP.
\end{proposition}
\begin{proof}
Since $A$ and $DAD$ have the same based graph, $A \circ X = O$ if and only if $DAD \circ X = O$.  On the other hand, $AX - XA = O$ if and only if 
\[
    DAXD - DXAD = DADDXD - DXDDAD = O.
\]
Also, $X = O$ if and only if $DXD = O$.  
\end{proof}

Thus, \cref{thm:iepstree} is naturally extended.

\begin{corollary}
\label{cor:iepstree}
Let $T$ be a tree and $\Lambda$ a multiset of real numbers.  Then $\Lambda$ is the spectrum of some matrix in $\mdso(T)$ with the SSP if and only if $\Lambda$ is the spectrum of some matrix in $\mptn(T)$ with the SSP.  
\end{corollary}

Next we see that the supergraph lemma naturally extends to matrices in $\mdso(G)$.  The only differences between \cref{thm:super} and \cref{thm:supersign} are that the new nonzero entries need to be negative and that the negative entries in the original matrix need to stay negative.  Indeed, these are natural consequence of the classical supergraph lemma.


\begin{theorem}[Supergraph lemma with signs]
\label{thm:supersign}
Let $H$ be a spanning supergraph of $G$.  Suppose $A\in\mptn(G)$ has the SSP.  Then there is a matrix $A'\in\mptn(H)$ such that $\spec(A') = \spec(A)$, $A'$ has the SSP, and $\|A' - A\|$ can be chosen arbitrarily small.  Moreover, the sign of each entry of $A'$ on $E(H) \setminus E(G)$ can be chosen arbitrarily, while other off-diagonal entries have the same signs as those of $A$ correspondingly.
\end{theorem}
\begin{proof}
The original proof of the supergraph lemma \cite{gSAP} actually says more than its statement.  According to its proof, the entries of $A'$ corresponding to $E(H) \setminus E(G)$ can be chosen to be arbitrary real numbers.  Thus, we may choose them to be small values with any prescribed signs.  Since $\|A' - A\|$ can be chosen to be small, off-diagonal entries other than $E(H) \setminus E(G)$ preserves their signs.  
\end{proof}

We remark that the fact that the new nonzero entries can be chosen arbitrarily is also used in \cite{LOS24} to guarantee a matrix with any prescribed spectrum and a generic eigenbasis.

By \cite[Theorem~34]{gSAP}, the direct sum $A\oplus B$ of two matrices $A$ and $B$ has the SSP if and only if both $A$ and $B$ have the SSP and they share no common eigenvalue.  This leads to an immediate application.  

\begin{example}
\label{ex:discrete}
Let $G$ be a graph on $n$ vertices and $\Lambda = \{\lambda_1, \ldots, \lambda_n\}$ a set of distinct real numbers.  By \cite[Theorem~34]{gSAP}, the diagonal matrix $D$ with diagonal entries $\Lambda$ has the SSP.  Since $G$ is a spanning supergraph of $\overline{K_n}$ and $D \in \mdso(\overline{K_n})$, there is a matrix $A'\in\mdso(G)$ with the SSP and $\spec(A') = \Lambda$. 
\end{example}

Next we review some background of the verification matrix to introduce the liberation lemma.  The SSP verification matrix $\Psi$ of a matrix $A$ was first introduced in \cite{IEPG2}.  It is known that the left kernel of $\Psi$ is trivial if and only if $A$ has the SSP.  We use the equivalent statement in \cite[Proposition~7.7]{IEPG2} as the definition.  Let $G$ be a graph on $n$ vertices and $M$ an $n\times n$ matrix.  We adopt the notation that $\mvec_{E(\overline{G})}(M)$ is the vector recording entries of $M$ corresponding to $\{(i,j): \{i,j\}\in E(\overline{G}),\ i < j\}$, and that $\mvec_{E_o}(M)$ is the vector recording entries of $M$ corresponding to $E_o = \{(i,j): 1\leq i < j \leq n\}$.  Here the entries are ordered lexicographically.

\begin{definition}[\cite{IEPG2}]
Let $G$ be a graph on $n$ vertices and $A\in\mptn(G)$.  The \emph{SSP verification matrix} of $A$ is a $|E(\overline{G})|\times \binom{n}{2}$ matrix whose rows are indexed by $\{(i,j): \{i,j\}\in E(\overline{G}),\ i < j\}$ in the lexicographical order such that the $(i,j)$-row is $\mvec_{E_o}(AX_{i,j} - X_{i,j}A)$.  Here $X_{i,j}$ is the symmetric matrix whose $i,j$-entry and $j,i$-entry are $1$ and other entries are $0$.
\end{definition}

Let $\mptncl(G)$ be the topological closure of $\mptn(G)$.  That is, $\mptncl(G)$ allows those entries corresponding to edges to be zero.  Moreover, let $\mptnclo(G)$ be the matrices in $\mptncl(G)$ with diagonal entries all zero.  Thus, the condition $A\circ X = O$ is equivalent to $X\in\mptncl(\overline{G})$ when $A\in\mptn(G)$, while $A\circ X = O$ and $I\circ X = O$ are equivalent to $X\in\mptnclo(G)$.  Based on this, Lin, Oblak, and Šmigoc~\cite{LiberationG} introduced an extension of the SSP. 
Let $G$ be a graph and $G'$ a spanning supergraph of $G$.  A matrix $A\in\mptn(G)$ is said to have the \emph{strong spectral property with respect to $G'$} if $X = O$ is the only real symmetric matrix that satisfies $X\in\mptnclo(\overline{G'})$ and $AX - XA = O$.  

One of the challenges for using the liberation lemma is to find the vector $\bx$.  Lin, Oblak, and Šmigoc~\cite{LiberationG} extended the notion of the SSP and introduced the liberation set, which is proved to be the same as $\vsupp(\bx)$ for some valid $\bx$.  

\begin{definition} Let $G$ be a graph and $A\in\mptn(G)$. A nonempty set of edges $\beta\subseteq E(\overline{G})$ is called an \emph{SSP liberation set} of $A$ if and
only if $A$ has the SSP with respect to $G + \beta'$ for all $\beta' \subset \beta$ with $|\beta'| = |\beta| - 1$.
\end{definition}

It is shown that \cite[Proposition~2.6]{LiberationG} $\beta$ is an SSP liberation set if and only if $\beta = \vsupp(\bx)$ for some vector $\bx$ valid for the liberation lemma (\cref{thm:lib}).

Let $A\in\mptn(G)$.  An \emph{annihilator} of $A$ is a matrix $X$ such that $A\circ X = I \circ X = O$ and $AX - XA = O$.  Define $\mathcal{R}(A)$ as the set of annihilators of $A$, which is a subspace in the space of symmetric matrices $\msym$. Many definitions and results can be rephrased in the language of annihilators.  

\begin{proposition}
\label{prop:ann}
Let $G$ be a graph, $A\in\mptn(G)$, and $\Psi$ the SSP verification matrix of $A$.  Let $G'$ be a spanning supergraph of $G$.  Then the following hold. 
\begin{enumerate}[label={\rm(\arabic*)}]
\item $A$ has the SSP if and only if $\mathcal{R}(A) = \{O\}$.
\item $A$ has the SSP with respect to $G'$ if and only if $\mathcal{R}(A) \cap \mptnclo(\overline{G'}) = \{O\}$.
\item $\beta$ is an SSP liberation set if and only if $\mathcal{R}(A) \cap \mptnclo(\overline{G + \beta'})= \{O\}$ for all $\beta'\subset\beta$ with $|\beta'| = |\beta| - 1$.
\item A real symmetric matrix $X\in\mptnclo(\overline{G})$ is in $\mathcal{R}(A)$ if and only if $\mvec_{E(\overline{G})}(X)\trans \Psi = \bzero\trans$.
\item A vector $\bx$ is in the column space $\Col(\Psi)$ if and only if $\inp{\bx}{\mvec_{E(\overline{G})}(X)} = 0$ for all $X\in\mathcal{R}(A)$.  
\end{enumerate}
\end{proposition}


\begin{theorem}[Liberation lemma with signs]
\label{thm:libsign}
Let $G$ be a graph and $A\in\mptn(G)$.  Suppose $\beta\subseteq E(\overline{G})$ is a liberation set and $\bx$ a vector with $\vsupp(\bx) = \beta$ such that $\inp{\bx}{\mvec_{E(\overline{G})}(X)} = 0$ for all $X\in\mathcal{R}(A)$.  Then there is a matrix $A'\in\mptn(G + \beta)$ such that $\spec(A') = \spec(A)$, $A'$ has the SSP, and $\|A' - A\|$ can  be chosen to be arbitrarily small.  Moreover, the sign pattern of $A'$ is the same as that of $\bx$ on $\beta$ and the same as that of $A$ on the off-diagonal entries outside $\beta$.
\end{theorem}
\begin{proof}
Let $\Psi$ be the SSP verification matrix of $A$.  Since $\beta$ is an SSP liberation set, rows of $\Psi$ outside $\beta$ are linearly independent by \cite[Proposition~2.6]{LiberationG}.  On the other hand, the $\bx$ with the given assumptions is a vector in $\Col(\Psi)$ by \cref{prop:ann}.  By the liberation lemma (\cref{thm:lib}), there is a matrix $A'\in\mptn(G + \beta)$ such that $\spec(A') = \spec(A)$, $A'$ has the SSP, and $\|A' - A\|$ can be chosen to be arbitrarily small.

According to the proof of the liberation lemma (\cref{thm:lib}), $A'$ is obtained by perturbing $A$ toward the direction $\bx$ on $\beta$.  Therefore, the sign pattern of $A'$ on $\beta$ is the same as that of $\bx$.  When the perturbation is small enough, all off-diagonal entries other than $\beta$ preserve the same sign pattern.
\end{proof}

\begin{example}
\label{ex:libcn}
This example constructs a matrix in $\mdso(C_n)$ with ordered multiplicity list $\oml = (m_1, \ldots, m_q)$ such that $m_k = 2$ for some even $k$ while $m_i = 1$ for all $i \neq k$.  

Let $A\in\mdso(P_{n-1})$ be a matrix with eigenvalues $\lambda_1 < \cdots < \lambda_{n-1}$ constructed by \cref{thm:pn,thm:iepstree}.  Let $\bv_1, \ldots, \bv_{n-1}$ be the corresponding eigenvector.  It is known that the product of the first entry and the last entry of $\bv_i$ has the sign $(-1)^{i-1}$; see, e.g., \cite{Ferguson80, DPTZ19}.

Let $A_\lambda = A\oplus \begin{bmatrix}\lambda\end{bmatrix}$.  Thus, the spectrum of $A_\lambda$ is the spectrum of $A$ added with $\lambda$.  Since every matrix in $\mptn(P_{n-1})$ and $\mptn(K_1)$ has the SSP by \cref{rem:ssppc}, we know 
\[
    \mathcal{R}(A_{\lambda_i}) = \vspan\left\{\begin{bmatrix} O_{n-1} & \bv_i \\ \bv_i\trans & 0 \end{bmatrix}\right\}
\]
is a one-dimensional subspace.  By \cref{prop:ann}, $\beta = \{\{1,n\},\{n-1,n\}\}$ is an SSP liberation set of $A_{\lambda_i}$.  When $i$ is even and the product of the first entry and the last entry of $\bv_i$ is negative, there is a vector $\bx$ with $\vsupp(\bx) = \beta$ that is negative on $\beta$ and zero otherwise, and $\inp{\bx}{\mvec_{E(\overline{G})}(X)} = 0$ for all $X\in\mathcal{R}(A)$.  By \cref{thm:libsign}, there is a matrix $A'\in\mdso(C_n)$ with the SSP such that $\spec(A') = \spec(A)$.
\end{example}

\begin{corollary}
\label{cor:firework}
Let $H$ be a connected graph and $B\in\mdso(H)$ with the SSP and $\spec(B) = \{\lambda_1, \lambda_2^{(m_2)}, \ldots, \lambda_q^{(m_q)}\}$.  Let $G$ be the graph obtained from $P_p\dunion H$ by joining one endpoint of the path $P_p$ with all vertices of $H$, where $p \geq q-1$.  Then there is a matrix $A'\in\mdso(G)$ with the SSP such that 
\[
    \spec(A') = \{\lambda_1, \lambda_2^{(m_2+1)}, \ldots, \lambda_q^{(m_q+1)}, \mu_1, \ldots, \mu_{p-q+1}\},
\]
where $\lambda_i$'s and $\mu_i$'s are all distinct.
\end{corollary}
\begin{proof}
Since $H$ is connected, we may assume $\lambda_1$ is the smallest eigenvalue.  Let $\{\bb_1, \ldots, \bb_h\}$ be an orthonormal eigenbasis of $H$ such that $\bb_1$ is entrywise negative, where $h$ is the number of vertices of $H$.  

Let $\Omega = \{\lambda_2, \ldots, \lambda_q, \mu_1, \ldots, \mu_{p-q+1}\}$ be a set of distinct real numbers.  By by \cref{thm:pn,thm:iepstree}, find a matrix $A\in\mdso(P_p)$ with the spectrum $\Omega$. 
Necessarily, $A$ has the SSP by \cref{rem:ssppc}.  Let $\ba_2, \ldots, \ba_q$ be the eigenvectors $A$ with respect to $\lambda_2, \ldots, \lambda_q$, respectively.  We may assume that $\{\ba_2, \ldots, \ba_q\}$ is orthonormal.  

Since both $A$ and $B$ have the SSP, $\mathcal{R}(A\oplus B)$ is an $(h-1)$-dimensional subspace spanned by elements of the form $f(\ba_i\bb_j\trans)$ such that $\ba_i$ and $\bb_j$ correspond to the same eigenvalue, where 
\[
    f(Y) = \begin{bmatrix}
    O & Y \\
    Y\trans & O
    \end{bmatrix}.
\]
Let $\beta$ be the set of edges from the last vertex of $P_p$ to every vertex of $H$, which corresponds to those entries in the bottom row of $Y$.   Note that each of $\bb_2, \ldots, \bb_h$ appears exactly once in the spanning set.  Also, by zero forcing the last entry of $\ba_i$ is always nonzero; see, e.g., \cite{IEPG2}.  Thus, $\mvec_{\beta}(R)$ for $R\in\mathcal{R}(A\oplus B)$ are nonzero multiple of $\bb_2, \ldots, \bb_h$, respectively.  Let $\bx$ be the vector indexed by $E(\overline{P_p\dunion H})$ and whose entries corresponding to $\beta$ are $\bb_1$ and zero otherwise.  Then $\inp{\bx}{\mvec_{E(\overline{P_p\dunion H})}(R)} = 0$ for all $R\in\mathcal{R}(A\oplus B)$.  

On the other hand, we claim that $\beta$ is an SSP liberation set.  Let $Q$ be the matrix whose rows are $\bb_1, \ldots, \bb_h$.  Since $\{\bb_1, \ldots, \bb_h\}$ is an orthonormal eigenbasis of $H$, the matrix $Q$ is orthogonal.  This implies $Q\trans = Q^{-1}$ and $\det(Q) = \pm 1$.  With $Q\trans = Q^{-1} = \pm Q\adj$, we know all $(h-1)\times (h-1)$ minors of $Q$ without using the row of $\bb_1$ are nonzero.  Therefore, by vanishing any $h-1$ entries on $\beta$, the only remaining matrix in $\mathcal{R}(A\oplus B)$ is $O$.  That is, $\beta$ is a liberation set by \cref{prop:ann}.  By \cref{thm:libsign}, there is a matrix $A'$ with the desired spectrum, with the SSP, and $A'\in\mdso(G)$.  
\end{proof}

We end this section by observing that the bifurcation lemma works well with signs.  \cref{thm:bifsign} is an immediate consequence of the classical bifurcation lemma \cref{thm:bif}, since sufficiently small perturbations preserve the sign pattern of off-diagonal entries.

\begin{theorem}[Bifurcation lemma with signs]
\label{thm:bifsign}
Let $G$ be a graph and $A\in\mptn(G)$ with the SSP.  Then there is $\epsilon > 0$ such that for any $M$ with $\|M - A\| < \epsilon$, a matrix $A'\in\mptn(G)$ exists with the SSP and $\spec(A') = \spec(M)$.  Moreover, $A'$ and $A$ have the same off-diagonal sign pattern.
\end{theorem}

\section{Small graphs}
\label{sec:small}

We first note that \cref{thm:pn}, \cref{rem:ssppc}, and \cref{thm:iepstree} lead to the conclusion that $\Lambda$ is the spectrum of $A\in\mdso(P_n)$ if and only if elements in $\Lambda$ are all distinct.  Moreover, all matrices in $\mdso(P_n)$ has the SSP.

On the other hand, \cref{thm:kn} states that $\Lambda$ is the spectrum of $A\in\mdso(K_n)$ if and only if the smallest value in $\Lambda$ is simple.  When $A\in\mdso(K_n)$, $A\circ X = O$ and $I\circ X = O$ imply $X = O$, so every matrix $A\in\mdso(K_n)$ has the SSP.  

For $n \leq 3$, all graphs are either complete graphs or paths, and the IEPS for them have been solved.

For graphs on $4$ or $5$ vertices, \cite[Fig.~1]{IEPG2} shows all the possible ordered multiplicity lists, and the paper proved that they are spectrally arbitrary.  Since $\mdso(G)\subseteq\mptn(G)$, the allowed spectra in $\mdso(G)$ are fewer than those of $\mptn(G)$, yet we know by \cref{cor:iepstree} the two answers are the same when $G$ is a tree.

\subsection{Graphs with \texorpdfstring{$n = 4$}{n = 4} vertices}
\label{ssec:graph4}

\begin{table}[t]
\begin{center}
\begin{tabular}{|>{\centering}m{2cm}|c|>{\small}m{8cm}|}
\hline
\putpic{p4pic} & $P_4$ & \mtt[$\overline{K_4}$]{1111} \\ \hline
\putpic{k13pic} & $K_{1,3}$ & \mtt[$\overline{K_4}$]{1111}, \mtt[$T$]{121} \\ \hline
\putpic{pawpic} & $\Paw$ & \mtt[$\overline{K_4}$]{1111}, \mtt[$K_1\dunion K_3$]{121}, \mtt[$K_1\dunion K_3$]{112}, \ntt[PF]{211} \\ \hline
\putpic{c4pic} & $C_4$ & \mtt[$\overline{K_4}$]{1111}, \mtt[lib]{121}, \ntt[Ferguson]{112}, \ntt[PF]{211}, \ntt[PF]{22} \\ \hline
\putpic{dmndpic} & $\Dmnd$ & \mtt[$\overline{K_4}$]{1111}, \mtt[$K_1\dunion K_3$]{121}, \mtt[$K_1\dunion K_3$]{112}, \ntt[PF]{211}, \ntt[PF]{22} \\ \hline
\putpic{k4pic} & $K_4$ & \mtt[$\overline{K_4}$]{1111}, \mtt[$K_1\dunion K_3$]{121}, \mtt[$K_1\dunion K_3$]{112}, \mtt[$K$]{13}, \ntt[PF]{211}, \ntt[PF]{22}, \ntt[PF]{31} \\ \hline
\end{tabular}
\end{center}
\caption{Ordered multiplicity lists for graphs on $4$ vertices, where the lists with strikethrough are allowed in $\mptn(G)$ but not in $\mdso(G)$.}
\label{tbl:graph4}
\end{table}

\cref{tbl:graph4} includes all graphs on $4$ vertices and all the ordered multiplicity lists allowed in $\mptn(G)$ for each graph, where the lists with strikethrough are forbidden in $\mdso(G)$.  

Here the subscript gives the reason of why a list is allowed or not.
\begin{itemize}
\item A graph name in the subscript means the spectrum is inherited from the graph by the supergraph lemma (\cref{thm:supersign}); see, e.g., \cref{ex:discrete}.
\item $T$ means the graph is a tree, so all spectra allowed in $\mptn(G)$ are also allowed in $\mdso(G)$ by \cref{cor:iepstree}.
\item The subscript lib means the spectrum is allowed by the liberation lemma with signs (\cref{thm:libsign}); see, e.g., \cref{ex:libcn}.
\item Ferguson means the ordered multiplicity list is forbidden by Ferguson's result on periodic Jacobi matrix (\cref{thm:cn}).  
\item PF means the ordered multiplicity list is forbidden by the Perron--Frobenius theorem (\cref{obs:pf}).
\item $K$ means the graph is a complete graph, so the answer is given by \cref{thm:kn}.
\end{itemize}

These arguments completely answer the IEPS for graphs on $4$ vertices.  We say an ordered multiplicity list is \emph{spectrally arbitrary with the SSP} if every spectrum with the ordered multiplicity list can be realized by a matrix with the SSP.

\begin{theorem}
\label{thm:graph4}
The lists in \cref{tbl:graph4} without strikethrough are all the ordered multiplicity lists allowed in $\mdso(G)$ for each graph $G$, and these ordered multiplicity list are spectrally arbitrary with the SSP.
\end{theorem}

\subsection{Graphs with \texorpdfstring{$n = 5$}{n = 5} vertices}
\label{ssec:graph5}

\begin{table}[t!]
\begin{center}
\begin{tabular}{|>{\centering}m{2cm}|c|>{\small}m{8cm}|}
\hline
\putpic{p5pic} & $P_5$ & \mtt[$\overline{K_5}$]{11111} \\ \hline 
\putpic{s211pic} & $S(2,1,1)$ & \mtt[$\overline{K_5}$]{11111}, \mtt[$K_1\dunion K_{1,3}$]{1211}, \mtt[$K_1\dunion K_{1,3}$]{1121} \\ \hline 
\putpic{k14pic} & $K_{1,4}$ & \mtt[$\overline{K_5}$]{11111}, \mtt[$K_1\dunion K_{1,3}$]{1211}, \mtt[$K_1\dunion K_{1,3}$]{1121}, \mtt[$T$]{131} (no SSP) \\ \hline 
\putpic{l32pic} & $L(3,2)$ & \mtt[$\overline{K_5}$]{11111}, \mtt[$K_1\dunion K_{1,3}$]{1211}, \mtt[$K_1\dunion K_{1,3}$]{1121}, \mtt[$2K_1\dunion K_3$]{1112}, \ntt[PF]{2111} \\ \hline 
\putpic{bullpic} & $\Bull$ & \mtt[$\overline{K_5}$]{11111}, \mtt[$K_1\dunion K_{1,3}$]{1211}, \mtt[$K_1\dunion K_{1,3}$]{1121}, \mtt[$2K_1\dunion K_3$]{1112}, \ntt[PF]{2111} \\ \hline 
\putpic{c5pic} & $C_5$ & \mtt[$\overline{K_5}$]{11111}, \mtt[lib]{1211}, \mtt[lib]{1112}, \mtt[oc]{122}, \ntt[Ferguson]{1121}, \ntt[PF]{2111}, \ntt[PF]{221} \\ \hline 
\putpic{camppic} & $\Camp$ & \mtt[$\overline{K_5}$]{11111}, \mtt[$K_1\dunion K_{1,3}$]{1211}, \mtt[$K_1\dunion K_{1,3}$]{1121}, \mtt[$2K_1\dunion K_3$]{1112}, \mtt[oc]{122}, \ntt[PF]{2111}, \ntt[PF]{221} \\ \hline 
\putpic{bnrpic} & $\Bnr$ & \mtt[$\overline{K_5}$]{11111}, \mtt[$K_1\dunion K_{1,3}$]{1211}, \mtt[$K_1\dunion K_{1,3}$]{1121}, \ntt[bip]{1112}, \ntt[PF]{2111}, \ntt[bip]{122}, \ntt[PF]{221}, \ntt[PF]{212} \\ \hline 
\putpic{dartpic} & $\Dart$ & \mtt[$\overline{K_5}$]{11111}, \mtt[$K_1\dunion K_{1,3}$]{1211}, \mtt[$K_1\dunion K_{1,3}$]{1121}, \mtt[$2K_1\dunion K_3$]{1112}, \mtt[$\Camp$]{122}, \ntt[PF]{2111}, \ntt[PF]{221}, \ntt[PF]{212} \\ \hline 
\putpic{kitepic} & $\Kite$ & \mtt[$\overline{K_5}$]{11111}, \mtt[$K_1\dunion K_{1,3}$]{1211}, \mtt[$K_1\dunion K_{1,3}$]{1121}, \mtt[$2K_1\dunion K_3$]{1112}, \ntt[PF]{2111}, \ntt[diam]{122}, \ntt[PF]{221}, \ntt[PF]{212} \\ \hline 
\putpic{hspic} & $\Hs$ & \mtt[$\overline{K_5}$]{11111}, \mtt[$K_1\dunion K_{1,3}$]{1211}, \mtt[$K_1\dunion K_{1,3}$]{1121}, \mtt[$2K_1\dunion K_3$]{1112}, \mtt[$C_5$]{122}, \ntt[PF]{2111}, \ntt[PF]{221}, \ntt[PF]{212} \\ \hline 
\putpic{gempic} & $\Gem$ & \mtt[$\overline{K_5}$]{11111}, \mtt[$K_1\dunion K_{1,3}$]{1211}, \mtt[$K_1\dunion K_{1,3}$]{1121}, \mtt[$2K_1\dunion K_3$]{1112}, \mtt[$C_5$]{122}, \ntt[PF]{2111}, \ntt[PF]{221}, \ntt[PF]{212} \\ \hline 
\putpic{bflypic} & $\Bfly$ & \mtt[$\overline{K_5}$]{11111}, \mtt[$K_1\dunion K_{1,3}$]{1211}, \mtt[$K_1\dunion K_{1,3}$]{1121}, \mtt[$2K_1\dunion K_3$]{1112}, \mtt[$\Camp$]{122}, \ntt[PF]{2111}, \ntt[PF]{221}, \ntt[PF]{212}, \ntt[PF]{311}, \ntt[cut]{131}, \mtt[IEPG2]{113} (no SSP) \\ \hline 
\end{tabular}
\end{center}
\caption{Ordered multiplicity lists for graphs on $5$ vertices not containing $K_{2,3}$ nor $K_4$ as a subgraph, where the lists with strikethrough are allowed in $\mptn(G)$ but not in $\mdso(G)$.}
\label{tbl:graph5wo3}
\end{table}

\begin{table}[t!]
\begin{center}
\begin{tabular}{|>{\centering}m{2cm}|c|>{\small}m{8cm}|}
\hline
\putpic{k23pic} & $K_{2,3}$ & \mtt[$\overline{K_5}$]{11111}, \mtt[$K_1\dunion K_{1,3}$]{1211}, \mtt[$K_1\dunion K_{1,3}$]{1121}, \mtt[IEPG2]{131}, \ntt[PF]{2111}, \ntt[bip]{1112}, \ntt[bip]{122}, \ntt[PF]{221}, \ntt[PF]{212} \\ \hline 
\putpic{t5pic} & $T_5$ & \mtt[$\overline{K_5}$]{11111}, \mtt[$K_1\dunion K_{1,3}$]{1211}, \mtt[$K_1\dunion K_{1,3}$]{1121}, \mtt[$2K_1\dunion K_3$]{1112}, \mtt[$\Camp$]{122}, \mtt[$K_{2,3}$]{131}, \ntt[PF]{2111}, \ntt[PF]{221}, \ntt[PF]{212} \\ \hline 
\putpic{l41pic} & $L(4,1)$ & \mtt[$\overline{K_5}$]{11111}, \mtt[$K_1\dunion K_{1,3}$]{1211}, \mtt[$K_1\dunion K_{1,3}$]{1121}, \mtt[$2K_1\dunion K_3$]{1112}, \mtt[$\Camp$]{122}, \mtt[$K_1\dunion K_4$]{131}, \mtt[$K_1\dunion K_4$]{113}, \ntt[PF]{2111}, \ntt[PF]{221}, \ntt[PF]{212}, \ntt[PF]{311} \\ \hline 
\putpic{k4epic} & $(K_4)_e$ & \mtt[$\overline{K_5}$]{11111}, \mtt[$K_1\dunion K_{1,3}$]{1211}, \mtt[$K_1\dunion K_{1,3}$]{1121}, \mtt[$2K_1\dunion K_3$]{1112}, \mtt[$C_5$]{122}, \mtt[$K_{2,3}$]{131}, \ntt[PF]{2111}, \ntt[PF]{221}, \ntt[PF]{212}, \ntt[PF]{311}, \ntt[bip]{113}, \ntt[PF]{32}, \ntt[PF]{23} \\ \hline 
\putpic{w5pic} & $W_5$ & \mtt[$\overline{K_5}$]{11111}, \mtt[$K_1\dunion K_{1,3}$]{1211}, \mtt[$K_1\dunion K_{1,3}$]{1121}, \mtt[$2K_1\dunion K_3$]{1112}, \mtt[$C_5$]{122}, \mtt[$K_{2,3}$]{131}, \ntt[PF]{2111}, \ntt[PF]{221}, \ntt[PF]{212}, \ntt[PF]{311}, \ntt[bip]{113}, \ntt[PF]{32}, \ntt[PF]{23} \\ \hline 
\putpic{fhspic} & $\FHs$ & \mtt[$\overline{K_5}$]{11111}, \mtt[$K_1\dunion K_{1,3}$]{1211}, \mtt[$K_1\dunion K_{1,3}$]{1121}, \mtt[$2K_1\dunion K_3$]{1112}, \mtt[$C_5$]{122}, \mtt[$K_1\dunion K_4$]{131}, \mtt[$K_1\dunion K_4$]{113}, \ntt[PF]{2111}, \ntt[PF]{221}, \ntt[PF]{212}, \ntt[PF]{311}, \ntt[PF]{32}, \ntt[PF]{23} \\ \hline 
\putpic{k5epic} & $K_5 - e$ & \mtt[$\overline{K_5}$]{11111}, \mtt[$K_1\dunion K_{1,3}$]{1211}, \mtt[$K_1\dunion K_{1,3}$]{1121}, \mtt[$2K_1\dunion K_3$]{1112}, \mtt[$C_5$]{122}, \mtt[$K_1\dunion K_4$]{131}, \mtt[$K_1\dunion K_4$]{113}, \ntt[PF]{2111}, \ntt[PF]{221}, \ntt[PF]{212}, \ntt[PF]{311}, \ntt[PF]{32}, \ntt[PF]{23} \\ \hline 
\putpic{k5pic} & $K_5$ & \mtt[$\overline{K_5}$]{11111}, \mtt[$K_1\dunion K_{1,3}$]{1211}, \mtt[$K_1\dunion K_{1,3}$]{1121}, \mtt[$2K_1\dunion K_3$]{1112}, \mtt[$C_5$]{122}, \mtt[$K_1\dunion K_4$]{131}, \mtt[$K_1\dunion K_4$]{113}, \mtt[$K$]{14}, \ntt[PF]{2111}, \ntt[PF]{221}, \ntt[PF]{212}, \ntt[PF]{311}, \ntt[PF]{32}, \ntt[PF]{23}, \ntt[PF]{41} \\ \hline
\end{tabular}
\end{center}
\caption{Ordered multiplicity lists for graphs on $5$ vertices containing $K_{2,3}$ or $K_4$ as a subgraph, where the lists with strikethrough are allowed in $\mptn(G)$ but not in $\mdso(G)$.}
\label{tbl:graph5w3}
\end{table}

\cref{tbl:graph5w3,tbl:graph5wo3} include all graphs on $5$ vertices and all ordered multiplicity lists allowed in $\mptn(G)$ for each graph, where the lists with strikethrough are forbidden in $\mdso(G)$.  

Similarly, the subscript gives the reason of why a list is allowed or not.  Some of them are the same as those in \cref{ssec:graph4}, so we only list the new arguments.
\begin{itemize}
\item The subscript bip means the spectrum is forbidden by \cref{thm:bip} or \cref{cor:bipv}; see, e.g., \cref{ex:k23bip,ex:k4se}.
\item The subscript oc means the the spectrum is allowed in $\mptn(G)$ by \cite{IEPG2} and, moreover, the realizing matrix has to be in $\mdso(G)$ according \cref{thm:unico}; see, e.g., \cref{ex:unic5}.
\item The subscript diam means the spectrum is forbidden because of the diameter bound (\cref{thm:diam}); see, e.g., \cref{ex:kited}.
\item The subscript cut means the spectrum is forbidden in $\mdso(G)$ by \cref{lem:minup}; see, e.g., \cref{ex:bfly}.
\item IEPG2 means that the realizing matrix in $\mptn(G)$ given in \cite{IEPG2} is indeed in $\mdso(G)$; see, e.g., \cref{ex:bflybu,ex:k23131}.
\end{itemize}

\begin{example}
\label{ex:bflybu}
Consider the graph $\Bfly = K_1 \vee (2K_2)$.  By \cite{IEPG2,Kempton10}, one may construct a matrix as follows.  Let 
\[
    A = \begin{bmatrix}
        x & x \\
        \frac{1}{\sqrt{2}} & 0 \\
        \frac{1}{\sqrt{2}} & 0 \\
        0 & \frac{1}{\sqrt{2}} \\
        0 & \frac{1}{\sqrt{2}} \\
    \end{bmatrix}
    \text{ with }x > 0.
\]
Then $-AA\trans\in\mdso(\Bfly)$ and 
\[
    -A\trans A = \begin{bmatrix}
        -x^2 - 1 & -x^2 \\
        -x^2 & -x^2 - 1
    \end{bmatrix},
\]
which has the spectrum $\{-2x^2 - 1, -1\}$.  Since $-AA\trans$ and $-A\trans A$ have the same nonzero eigenvalues, $\spec(-AA\trans) = \{-2x^2 - 1, -1, 0, 0, 0\}$.  By choosing appropriate $x$ and scale and shift, $-pAA\trans + qI$ realizes all possible spectrum with the ordered multiplicity list \mtt{113}.  By \cite{IEPG2}, any matrix with such ordered multiplicity list does not have the SSP.
\end{example}

\begin{example}
\label{ex:k23131}
By \cite{IEPG2}, the matrix 
\[
    M = \begin{bmatrix}
        -a & 0 & -b & -b & -b \\
        0 & ac^2 & -bc & -bc & -bc \\
        -b & -bc & 0 & 0 & 0 \\
        -b & -bc & 0 & 0 & 0 \\
        -b & -bc & 0 & 0 & 0 \\
    \end{bmatrix}
    \text{ with } b\neq 0,\ c\neq 0,\pm1
\]
realizes arbitrary spectrum with the ordered multiplicity list \mtt{131} by choosing $a,b,c$ and scale and shift; moreover, it has the SSP regardless the choice of $a,b,c$.  It is also straight forward to verify that $M\in\mdso(K_{2,3})$.  
\end{example}

These arguments completely answer the IEPS for graphs on $5$ vertices.

\begin{theorem}
\label{thm:graph5}
The lists in \cref{tbl:graph5w3,tbl:graph5wo3} without strikethrough are all the ordered multiplicity lists for each graph, and these ordered multiplicity list are spectrally arbitrary with the SSP unless it is marked with ``(no SSP)''.  If an ordered multiplicity list is marked with ``(no SSP)'', then any matrix in $\mdso(G)$ realizing the ordered multiplicity list does not have the SSP.
\end{theorem}

\section{Concluding Remarks}

Recall that a \emph{$Z$-matrix} is a real square matrix whose off-diagonal entries are nonpositive; see, e.g., \cite{BP94}.  Matrices in $\mdso(G)$ are precisely the symmetric $Z$-matrices whose off-diagonal zero-nonzero pattern is described by the graph $G$.  
Through the data in \cref{sec:small}, we see that the inverse problems on $\mptn(G)$ and that on $\mdso(G)$ behave differently in many different aspects, beyond the basic properties of a $Z$-matrix, such as \cref{obs:pf,thm:diam}. 
For example, \mtt{122} is allowed in $\mptn(\Bnr)$, $\mptn(\Dart)$, and $\mptn(\Kite)$; in contrast, it is allowed only in $\mdso(\Dart)$ yet forbidden in $\mdso(\Bnr)$ and $\mdso(\Kite)$ for different reasons.  Cycles remain a mysterious case where the ordered multiplicity list $\mathtt{1\cdots 2\cdots 1}$ is allowed if and only if \mtt{2} occurs at the even position by \cref{thm:cn,ex:libcn}.  Such restriction is not observed on other graphs, and it is worthy to explore more hidden properties from a cycle.  Notably, whether all allowed spectrum in $\mptn(C_n)$ can be realized by an matrix with the SSP in $\mptn(C_n)$ remains an open problem.  Similarly, whether all spectrum satisfying \cref{thm:cn} can be realized by a matrix with the SSP in $\mdso(G)$ is of interest as well.

On the other hand, the SSP again plays an important role in the study of IEPS.  There are other strong properties, such as the strong Arnold property (SAP) and the strong multiplicity property (SMP); see, e.g., \cite{gSAP,IEPG2,IEPGZF22}.  We note that the supergraph lemma with signs (\cref{thm:supersign}), the liberation lemma with signs (\cref{thm:libsign}), and the bifurcation lemma with signs (\cref{thm:bifsign}) naturally extend to SAP and SMP, so there are more potentials for these signed versions.

\section*{Acknowledgements}
A. Laikhuram and J. C.-H. Lin were supported by the National Science and Technology Council of Taiwan (grant no.\ NSTC-112-2628-M-110-003 and grant no.\ NSTC-113-2115-M-110-010-MY3).



\end{document}